\documentclass[12pt,oneside]{amsart}

\usepackage{amsthm,amsmath,amssymb}
\usepackage{mathrsfs}
\usepackage{enumerate}
\usepackage{amsfonts}
\usepackage{verbatim}
\usepackage{amsbsy}
\usepackage{amsmath}
\usepackage{amssymb}
\usepackage{MnSymbol}
\usepackage{mathrsfs}
\usepackage{xcolor}

\newtheorem{theorem}{Theorem}[section]
\newtheorem{lemma}[theorem]{Lemma}
\newtheorem{proposition}[theorem]{Proposition}
\newtheorem{conjecture}[theorem]{Conjecture}

\newtheorem*{claim}{Claim}
\newtheorem*{question}{Question}
 
\theoremstyle{definition}
\newtheorem{definition}[theorem]{Definition}

\theoremstyle{remark}
\newtheorem{remark}[theorem]{Remark}
\newtheorem{example}[theorem]{Example}

\DeclareMathOperator{\adm}{adm}
\DeclareMathOperator{\pre}{pre}

\DeclareMathOperator{\FR}{FR}

\DeclareMathOperator{\FS}{FS}

\DeclareMathOperator{\comp}{comp}
\DeclareMathOperator{\fib}{fib}
\DeclareMathOperator{\cyc}{cyc}


\begin{document}

\title{Ramsey Algebras and the Existence of Idempotent Ultrafilters
}
\author{Wen Chean Teh}
\address{The Ohio State University, Columbus, OH 43210 United States}
\address{School of Mathematical Sciences\\
Universiti Sains Malaysia\\
11800 USM, Malaysia}
\email{dasmenteh@usm.my}

\begin{abstract}
Hindman's Theorem says that every finite coloring of the positive natural numbers has a monochromatic set of finite sums. 
Ramsey algebras, recently introduced, are structures that satisfy an analogue of Hindman's Theorem.
It is an open problem posed by Carlson whether every Ramsey algebra has an idempotent ultrafilter. 
This paper developes a general framework to study idempotent ultrafilters.
Under certain countable setting, the main result roughly says that every nondegenerate Ramsey algebra has a nonprincipal idempotent ultrafilter in some nontrivial countable field of sets. 
This amounts to a positive result that addresses Carlson's question in some way. 
\end{abstract}

\maketitle

\section{introduction}

The set of natural numbers $\{0,1,2,\dotsc\}$ is denoted by $\omega$. Suppose $\langle x_i \rangle_{i \in \omega}$ is a sequence of natural numbers. Let $\FS(\langle x_i \rangle_{i \in \omega} )$ denote the set  $\{\, \sum_{i \in F} x_i \mid F \in\mathcal{P}_f(\omega)\backslash \{\emptyset\}\,\}$, where $\mathcal{P}_f(\omega)$ is the set of all finite subsets of $\omega$. Hindman's Theorem \cite{nH74} says that 
for every finite partition of the set of positive natural numbers $\mathbb{N}=X_0\cupdot X_1\cupdot \dotsb \cupdot X_N$, there exists a sequence $\langle x_i \rangle_{i \in \omega}$ of positive natural numbers such that $\FS(\langle x_i \rangle_{i \in \omega} )\subseteq X_j$ for some $0\leq j \leq N$.

A Ramsey algebra is an algebraic structure which possesses the property analogous to that possesed by the semigroup $(\omega, +)$ as in Hindman's Theorem. The formal definition of Ramsey algebras was suggested by Carlson and recently introduced by this author \cite{wcT13, wcT13a}.
The notion of Ramsey algebras is motivated by the study of Ramsey spaces, introduced by Carlson \cite{tC88} in 1988.
He showed that some space of infinite sequences of variable words over a finite alphabet, endowed with an analogous Ellentuck topology is a Ramsey space. This result, known as Carlson's Theorem in \cite[XVIII,~\S4]{HS12}, implies many earlier Ramsey theoretic results including Hindman's Theorem, Ellentuck's Theorem, the dual Ellentuck Theorem~\cite{CS84}, the Galvin-Prikry Theorem \cite{GP73} and the Hales-Jewett Theorem \cite{HJ63}. Due to his abstract version of Ellentuck's Theorem \cite{eE74}, this space is Ramsey because the corresponding algebra of variable words is a Ramsey algebra.  This connection between Ramsey algebras and Ramsey spaces was addressed in \cite{wcT13}.

In 1975 Galvin and Glazer (see \cite{wC77} or \cite{nH79}) 
gave a simple proof of Hindman's Theorem by showing the existence of idempotent ultrafilters. These are exactly the idempotent elements of the semigroup $(\beta\mathbb{N}, +)$, where $+$ is the extension of addition on $\mathbb{N}$ to $\beta\mathbb{N}$, the Stone-\v{C}ech compactification of $\mathbb{N}$. As $(\beta\mathbb{N}, +)$ is a compact right topological semigroup, by the Ellis-Numakura Lemma, it has an idempotent element. Later in 1987, under Martin's Axiom, Hindman \cite{nH87} showed the existence of a strongly summable ultrafilter, that is, an idempotent ultrafilter $U$ with the stronger property that it is generated by sets of finite sums.
Blass and Hindman \cite{BH87} later showed that their existence is independent of $\mathsf{ZFC}$.

Carlson's algebras of variable words are interesting Ramsey algebras that are not semigroups. The collection of operations in each of these algebras is finite but can be arbitrarily large depending on the size of the underlying finite alphabet. The nice interplay among the operations in a Ramsey algebra of variable words may provide insight into the problem of formation of new Ramsey algebras from semigroups or other elementary Ramsey algebras, which remains wide open. A key feature in Carlson's proof is the construction of certain ultrafilters idempotent for every operation in the corresponding algebra. These ultrafilters in turn allow the construction of homogeneous sequences, generalizing Galvin-Glazer proof of Hindman's Theorem.
Furthermore, assuming Martin's Axiom, this author \cite{wcT13b} has shown that  every nondegenerate Ramsey algebra has nonprincipal strongly reductible ultrafilters, that is,  analogues of strongly summable ultrafilters. This led Carlson to propose the following open problem.
\begin{question}[Carlson]
Can the existence of idempotent ultrafilters for a Ramsey algebra be proven in $\mathsf{ZFC}$?
\end{question}

This paper is motivated by the work of reverse mathematicians. In this area, the logical strength of true statements that can be formalized using the language of second order arithmetic are determined. To do that, particularly, field of sets and ultrafilters have to be encoded as sets of natural numbers. Hence, countability is a major consideration. 
Therefore, a general framework is proposed so that the notion of idempotent ultrafilters can be dealt with in this countable setting. Our main result indicates that a corresponding positive assertion pertaining to  Carlson's problem may be provable under the theory of  second order arithmetic.

\section{Preliminaries}

To us an \emph{algebra} is a pair $(A, \mathcal{F})$, where $A$ is a nonempty set and $\mathcal{F}$ is a collection of operations on $A$, none of which is nullary. ``Subalgebra"
will have its natural meaning.

Tuples are defined inductively using ordered pairs. The $1$-tuple $(x)$ is defined to be $x$. The $2$-tuple $(x,y)$ is the ordered pair of $x$ and $y$. If $n \geq 2$, then $(x_0,\dotsc,x_n)$ is the ordered pair $((x_0,\dotsc,x_{n-1}),x_n)$. 

The set of infinite and finite sequences in $A$ are denoted by ${^\omega}\!A$ and ${^{<\omega}}\!A$ res\-pectively.
Suppose $\vec{a}$ is an infinite sequence $\langle a_0, a_1,a_2,\dotsc\rangle$. For each $n \geq 1$,  let $\vec{a}\!\upharpoonright \! n$ denote the initial segment of $\vec{a}$ of length $n$, namely $\langle a_0, a_1,\dotsc, a_{n-1}\rangle$.  For every $n \in \omega$, the cut-off sequence $\langle a_n, a_{n+1},a_{n+2},\dotsc\rangle$ is denoted by $\vec{a}-n$. 
If $\vec{b}$ is a finite sequence $\langle b_0, b_1,\dotsc,b_{n-1}\rangle$, then $\vert \vec{b}\vert$ is the length of $\vec{b}$, and the concatenation $\vec{b}\ast \vec{a}$ of $\vec{b}$ and $\vec{a}$ is $\langle b_0, b_1,\dotsc,b_{n-1}, a_0, a_1,a_2,\dotsc\rangle$.

Suppose $f$ is an $n$-ary operation on a set $A$ and $\vec{a}$ is a finite sequence in $A$. We will write $f(\vec{a})$ for $f(\vec{a}(0), \dotsc,\vec{a}(n-1))$ for notational convenience, where it is implicitly assumed that the length of $\vec{a}$ is $n$.  Similarly, if $\bar{x}$---a symbol with a bar over it indicate a list---is a list of variables, we may write $f(\bar{x})$ and it is understood that the list $\bar{x}$ consists of $n$ variables.

A \emph{field of sets} over a set $S$ is a nonempty collection $\mathcal{A}$ of subsets of  $S$ that is closed under finite unions, finite intersections and complementation (relative to $S$). For our purposes, we will always assume that every field of sets contains all the singletons.

Suppose $\mathcal{A}$ is a field of sets over $S$. 
An \emph{ultrafilter} $U$ on $\mathcal{A}$ is a subset of $\mathcal{A}$ closed under intersections such that $S \in U$, $\emptyset \notin U$ and for every $X \in \mathcal{A}$, either $X\in U$ or $ X^c\in U$. Note that if $\mathcal{A}$ is the power set of $S$, an ultrafilter on $\mathcal{A}$ is the same as what is typically called an ultrafilter on $S$.
The set of ultrafilters on $S$ is denoted by $\beta S$.
The \emph{principal ultrafilter} $U$ on $\mathcal{A}$ \emph{generated} by an element $a$ in $S$ is the ultrafilter $\{\,X \in \mathcal{A}\mid a \in X\,\}$. 
 An ultrafilter $U$ on $\mathcal{A}$ is \emph{nonprincipal} if{f} it is not principal.
 Note that an ultrafilter $U$ on $\mathcal{A}$ over $S$ is nonprincipal if and only if every $X\in U$ is infinite. 

As an example, consider the collection $\mathcal{A}$ of all finite and cofinite subsets of a set $S$. Clearly, $\mathcal{A}$ is a field of sets over $S$. Additionally, 
the collection of all cofinite subsets of $S$ is the unique nonprincipal ultrafilter on $\mathcal{A}$.

Suppose $\{A_i\}_{i \in I}$ is an indexed collection of distinct sets. An $n$-ary operation on $\{A_i\}_{i \in I}$ is a function with domain $A_{i_1}\times \dotsb \times A_{i_n}$ for some $i_1,\dotsc, i_n\in I$ and codomain $A_j$ for some $j\in I$.
Suppose $\mathcal{F}$ is a collection of operations on  $\{ A_i  \}_{i \in I}$. The structure $(\{ A_i  \}_{i \in I}, \mathcal{F})$ is known as a \emph{heterogeneous algebra} in the literature (see \cite{BL70}).
Suppose $B_i \subseteq A_i$ for each $i\in I$. We say that $\{ B_i  \}_{i \in I}$ is closed under $\mathcal{F}$ if{f} for every $f\colon A_{i_1}\times \dotsb \times A_{i_n}\rightarrow A_{j}$ in $\mathcal{F}$, we have $f(x_1, \dotsc,x_n)\in B_j$ whenever $(x_1, \dotsc, x_n)\in B_{i_1}\times \dotsb \times B_{i_n}$.

The following lemma should be nothing new to the universal algebraists, logicians and perhaps mathematicians in general, especially for the case where $I$ is a singleton. Nevertheless, to our knowledge, the result in this form does not appear in the literature. It will be needed later and so we state it without proof. 

\begin{theorem}\label{2111}
Suppose $(\{ A_i  \}_{i \in I}, \mathcal{F})$ is a heterogeneous algebra, where $\mathcal{F}$ is countable. If $X_i$ is a countable subset of $A_i$ for each $i \in I$, then there exists a countable superset $B_i$ of $X_i$ for every $i \in I$ such that $\{ B_i  \}_{i \in I}$ is closed under $\mathcal{F}$.
\end{theorem}

\section{Ramsey Algebras and Idempotent Ultrafilters}\label{2706b}

A few terminologies are needed before we can introduce the notion of Ramsey algebra.
The following is a special type of composition of operations, to the author's knowledge, introduced by Carlson in  \cite{tC88}.
In fact, Carlson's definition is more general because it is defined for any heterogenoeus algebra.


\begin{definition}\label{0601a}
Suppose $(A, \mathcal{F})$ is an algebra. An operation $f \colon A^m\rightarrow A$ is an \emph{orderly composition} of $\mathcal{F}$ if{f} there exist  $g, h_1, \dotsc, h_n \in \mathcal{F}$ such that $f( \bar{x}_1, \dotsc, \bar{x}_n)=g(h_1(\bar{x}_1), \dotsc, h_n(\bar{x}_n))$.
We say that $\mathcal{F}$ is \emph{closed under orderly composition} if{f} $f \in \mathcal{F}$ whenever $f$ is an orderly composition of $\mathcal{F}$. The collection of \emph{orderly terms} over $\mathcal{F}$ is  the smallest collection of operations on $A$ that includes $\mathcal{F}$, contains the identity function on $A$ and is closed under orderly composition.
\end{definition}


\begin{definition}\label{0524c}
Suppose $(A, \mathcal{F})$ is an algebra and $\vec{a}, \vec{b}$ are infinite sequences in $A$.
We say that $\vec{a}$ is a \emph{reduction} of $\vec{b}$ with respect to $\mathcal{F}$, and write $\vec{a} \leq_{\mathcal{F}} \vec{b}$ if{f} there are finite sequences $\vec{b}_n$ and orderly terms $f_n$ over $\mathcal{F}$ for all
$n \in \omega$ such that $\vec{b}_0 \ast \vec{b}_1 \ast \vec{b}_2 \ast \dotsb$ is a subsequence of $\vec{b}$ and $\vec{a}(n)=f_n(\vec{b}_n)$ for all $n \in \omega$. 
\end{definition}

The following is an analogous definition of reduction defined for finite sequences.

\begin{definition}
Suppose $(A, \mathcal{F})$ is an algebra and $\vec{a}, \vec{b}$ are finite sequences in $A$. We say that $\vec{a}$ is a \emph{reduction} of $\vec{b}$ with respect to $\mathcal{F}$, and write $\vec{a} \unlhd_{\mathcal{F}} \vec{b}$  if{f} there are finite sequences $\vec{b}_n$  and orderly terms $f_n$ over $\mathcal{F}$ for $n< \vert \vec{a}\vert$ such that $\vec{b}_0 \ast \vec{b}_1 \ast \dotsb\ast \vec{b}_{\vert \vec{a}\vert -1}$ is a subsequence of $\vec{b}$ and $\vec{a}= \langle f_0(\vec{b}_0), f_1(\vec{b}_1), \dotsc,f_{\vert \vec{a}\vert -1}(\vec{b}_{\vert \vec{a}\vert -1})\rangle$.
\end{definition}

It is easy to check that $\leq_{\mathcal{F}}$ is a pre-partial ordering on ${^\omega}\!A$ and that $\unlhd_{\mathcal{F}}$ is a pre-partial ordering on ${^{<\omega}}\!A$.
 

\begin{definition}
Suppose $(A, \mathcal{F})$ is an algebra and $\vec{b}$ is an infinite sequence in $A$. An element $a$ of $A$ is a \emph{finite reduction} of $\vec{b}$ with respect to $\mathcal{F}$ if{f} $a$ is equal to $f(\vec{b}_0)$ for some orderly term $f$ over  $\mathcal{F}$  and some finite subsequence $\vec{b}_0$ of $\vec{b}$. The set of all finite reductions of $\vec{b}$ with respect to $\mathcal{F}$ is denoted by $\FR_{\mathcal{F}}(\vec{b})$.
\end{definition}


Our definitions of $\leq_{\mathcal{F}}$ and $\unlhd_{\mathcal{F}}$ are equivalent to a special case of the one given in \cite{tC88}, where the collection of operations contains all projections.
Our choice is driven by Hindman's Theorem; under our definition, $\FR_{\{+\}}(\vec{b})=\FS(\vec{b})$, where $\vec{b} \in  {^\omega}{\mathbb{N}}$.

Note that if $\vec{a} \leq_{\mathcal{F}} \vec{b}$, then $\FR_{\mathcal{F}}(\vec{a})\subseteq \FR_{\mathcal{F}}(\vec{b})$.

\begin{definition}
Suppose $(A, \mathcal{F})$ is an algebra. We say that $(A, \mathcal{F})$ is a \emph{Ramsey algebra} if{f}
for every $\vec{a}\in {^\omega}\!A$ and $X \subseteq A$, there exists $\vec{b} \leq_{\mathcal{F}} \vec{a}$ such that $\FR_{\mathcal{F}}(\vec{b})$ is either contained in or disjoint from $X$. 
\end{definition}

It is a consequence of Hindman's Theorem that every semigroup is a Ramsey algebra (see \cite[V,~\S2]{HS12}).

Suppose $(A, \mathcal{F})$ is an algebra such that for every $\vec{a} \in {^\omega}\!A$, there exists  $\vec{b} \leq_{\mathcal{F}} \vec{a}$ such that  $\vert \FR_{\mathcal{F}}(\vec{b}) \vert=1$. Then $(A, \mathcal{F})$ is trivially Ramsey, and we say that it is a \emph{degenerate Ramsey algebra}. 

\begin{theorem}[\cite{wcT13b}]\label{130613b}
Suppose $(A, \mathcal{F})$ is a nondegenerate Ramsey algebra. Then there exists $\vec{a}\in {^\omega}\! A$ such that $\FR_{\mathcal{F}}(\vec{b})$ is infinite whenever $\vec{b}\leq_{\mathcal{F}} \vec{a}$.
\end{theorem}


For the remaining of this section, some definitions and known results that will not be needed are briefly presented for the sake of comparison with our work in the new framework.

\begin{definition}\label{1207g}
Assume $U$ is an ultrafilter on $A$ and $f \colon A \rightarrow B$. Let $f_{\ast}(U)$ be the ultrafilter on $B$ defined by $\{\,X \subseteq B \mid f^{-1}[X] \in U \,\}$.
\end{definition}


\begin{definition}\label{0923a}
Assume $U$ and $V$ are ultrafilters on sets $A$ and $B$ respectively. $U \times V$ is the ultrafilter on $A\times B$ defined by
$$\{\,X \subseteq A\times B \mid \{\, a \in A \mid \{\,b\in B \mid (a,b) \in X \,\}\in V \,\}\in U\,\}$$
\end{definition}

Suppose $U_i$ is an ultrafilter on $A_i$ for each $i=1, \dotsc,n$ and $f \colon A_1 \times \dotsb \times A_n \rightarrow B$. We will write $f(U_1, \dotsc, U_n)$ for $f(U_1\times \dotsb \times U_n)$. 

\begin{definition}
Suppose $\mathcal{F}$ is a collection of operations  on a set $A$. An ultrafilter $U$ on $A$ is said to be \emph{idempotent} for $\mathcal{F}$ if{f} $f_*(U, \dotsc, U)=U$ whenever $f \in \mathcal{F}$.
\end{definition}

\begin{theorem}[\cite{tC88}]\label{1007a}
Suppose $\mathcal{F}$ is a collection of operations on a set $A$. If $U$ is an ultrafilter on $A$ idempotent for $\mathcal{F}$, then it is idempotent for the collection of orderly terms over $\mathcal{F}$.
\end{theorem}


Suppose $(A,f)$ is a semigroup.  Then $(\beta A, f_\ast)$ is a compact right topological semigroup. By the Ellis-Nakamura Lemma, there is an ultrafilter $U$ on $A$ idempotent for $f$.
An extensive exposition to the theory of compact right topological semigroup of this form can be found in \cite{HS12}.


The following is a generalization of strongly summable ultrafilter.

\begin{definition}[\cite{wcT13b}]\label{0207a}
Suppose $\mathcal{F}$ is a collection of operations on a set $A$ and $U$ is an ultrafilter on $A$. We say that $U$ is \emph{strongly reductible} for $\mathcal{F}$ if{f} for every $X \in U$, there exists $\vec{a} \in {^\omega}\!A$ such that $\FR_{\mathcal{F}}(\vec{a}) \subseteq X$ and $\FR_{\mathcal{F}}(\vec{a}-n) \in U$ for all $n \in \omega$.
\end{definition}

Every ultrafilter strongly reductible for $\mathcal{F}$ is necessarily idempotent for $\mathcal{F}$.

\begin{theorem}[\cite{wcT13b}]\label{0707b}
Assume Martin's Axiom. Suppose $(\omega, \mathcal{F})$ is a nondegenerate Ramsey algebra. Then there exists a nonprincipal strongly reductible ultrafilter.
\end{theorem}

\section{framework}

In this section, a framework will be proposed to study idempotent ultrafilters in some general setting.  We begin by generalizing Definition~\ref{1207g} and \ref{0923a} to field of sets. Later, under the new framework, results analogous to Theorem~\ref{1007a} and \ref{0707b} will be presented.

\begin{proposition}\label{1006a}
Suppose $\mathcal{A}$ is a field of sets over a set $S$ and $U$ is an ultrafilter on $\mathcal{A}$. Suppose $T$ is a set and $f \colon S\rightarrow T$.
If $\mathcal{B}$ is a field of sets over $T$ such that $f^{-1}[X]\in \mathcal{A}$ for all $X\in \mathcal{B}$, then the collection $\{\, X\in \mathcal{B}\mid f^{-1}[X]\in U\,\}$
is an ultrafilter on $\mathcal{B}$.
\end{proposition}

\begin{proof}
Straightforward.
\end{proof}

\begin{remark}
Consider the special case where $\mathcal{A}$ and $\mathcal{B}$ are the power sets $\mathcal{P}(S)$ and $\mathcal{P}(T)$. For every function $f \colon S\rightarrow T$,
it is automatic that $f^{-1}[X]\in \mathcal{P}(S)$ for all $X\in \mathcal{P}(T)$. In this case, the ultrafilter $\{\, X\in \mathcal{P}(T)\mid f^{-1}[X]\in U\,\}$ equals $f_*(U)$.
\end{remark}

\begin{theorem}\label{0915a}
Suppose $\mathcal{A}$ is a field of sets over $S$ and $\mathcal{B}$ is a field of sets over $T$. Suppose $U$ and $V$ are ultrafilters on $\mathcal{A}$ and $\mathcal{B}$  respectively. Suppose $\mathcal{C}$ is a field of sets over $S\times T$ such that
\begin{enumerate}
\item $\{\,t\in T \mid (s,t) \in X \,\}\in \mathcal{B}$ whenever $X \in \mathcal{C}$ and $s \in S$;
\item $\{\,s\in S \mid \{\,t\in T \mid (s,t) \in X \,\} \in V\,\} \in \mathcal{A}$ whenever $X \in \mathcal{C}$.
\end{enumerate}
Then $ \{\,X \in \mathcal{C} \mid \{\,s\in S \mid \{\,t\in T \mid (s,t) \in X \,\} \in V\,\} \in U\,\}$ is an ultrafilter on $\mathcal{C}$.
\end{theorem}

\begin{proof}
Straightforward. The first condition ensures that  
$\{\,s\in S \mid \{\,t\in T \mid (s,t) \in X^c \,\} \in V\,\} =\{\,s\in S \mid \{\,t\in T \mid (s,t) \in X \,\} \in V\,\}^c$
whenever $X \in \mathcal{C}$.
\end{proof}


\begin{definition}\label{1117a}
Suppose $\mathcal{A}, \mathcal{B}, \mathcal{C}, S, T, U,V$ are as stated in Theorem \ref{0915a}. Let $U \times_{ \mathcal{C}} V$ denote the ultrafilter on $\mathcal{C}$ given by
$\{\,X \in \mathcal{C} \mid \{\,s\in S \mid \{\,t\in T \mid (s,t) \in X \,\} \in V\,\} \in U\,\}$.
\end{definition}

Every subalgebra of a Ramsey algebra is Ramsey. In fact, it is easy to show that assuming the collection of operations is countable, an algebra is Ramsey if and only if every countable subalgebra is Ramsey. This is due to the fact that the subalgebra generated by the terms of a given infinite sequence is countable.
Furthermore, any ultrafilter idempotent for a subalgebra of a given algebra can be extended naturally to an ultrafilter idempotent for the algebra. 
Therefore, for our purposes, without loss of generality, we may assume the underlying set is always $\omega$.

Henceforth, we will focus on fields of sets over $\omega$ or Cartesian power of $\omega$. 
For the sake of readability, we assume the variables $a,b, a_0, a_1, x_0,x_1$ and et cetera run through $\omega$ whenever it is applicable; for example, $\{ \,( a_1,\dotsc, a_{n-1}) \mid \{\, a_n \mid  (a_1, \dotsc, a_n) \in X \,\}$ stands for $\{ \,( a_1,\dotsc, a_{n-1}) \in \omega^{n-1} \mid \{\, a_n \in \omega \mid  (a_1, \dotsc, a_n) \in X \,\}$.

\begin{definition}\label{1010a}
Suppose $\mathfrak{A}=\{ \mathcal{A}_n  \}_{n\geq 1}$ is an indexed collection such that $\mathcal{A}_n$ is a field of sets over $\omega^n$ for every $n\geq 1$. We say that
$\mathfrak{A}$ is \emph{admissible} if{f} for every $n \geq 2 $ and $X \in \mathcal{A}_{n}$
\begin{enumerate}
\item $ \{\,(a_2, \dotsc,a_{n}, a_1) \mid (a_1, \dotsc,a_{n})\in X \,\} \in \mathcal{A}_{n}$;
\item $ \{\,(a_2, \dotsc,a_{n}) \mid (a_1, \dotsc,a_{n})\in X \,\} \in \mathcal{A}_{n-1}$ for all $a_1 \in \omega$.
\end{enumerate}
\end{definition}



\begin{proposition}\label{1510a}
Suppose $\mathfrak{A}=\{ \mathcal{A}_n  \}_{n\geq 1}$ is admissible.
If $n \geq 2 $ and $X \in \mathcal{A}_{n}$, then $\square(k)$ holds for every $1 \leq k \leq n-1$, where $\square(k)$  stands for the statement:
$$\{\,(a_{n-k+1}, \dotsc,a_{n}) \mid (a_1, \dotsc,a_{n})\in X \,\} \in \mathcal{A}_k \text{ for every } a_1, \dotsc, a_{n-k} \in \omega.$$
\end{proposition}

\begin{proof}
Suppose $n \geq 2 $ and $X \in \mathcal{A}_{n}$.
We will prove by inverse induction on $k$ that $\square(k)$ holds for every $1\leq k \leq n-1$. By admissibility,
$\square(n-1)$ holds. Assume $\square(k)$ holds and $k>1$.  
Fix $a_1, \dotsc, a_{n-k+1} \in \omega$. 
Let $Y$ be the set $\{\,(a_{n-k+2}, \dotsc, a_{n}) \mid (a_1, \dotsc,a_{n-k+1}, a_{n-k+2}, \dotsc, a_n)\in X \,\}$
and let $Z$ be the set $\{\,(b,a_{n-k+2}, \dotsc,a_{n}) \mid (a_1,\dotsc, a_{n-k}, b, a_{n-k+2}, \dotsc,a_{n})\in X \,\}$. By the induction hypothesis,
$Z$ is in $\mathcal{A}_k$. Since 
$Y$ is equal to $\{\,(a_{n-k+2}, \dotsc,a_{n}) \mid (a_{n-k+1}, a_{n-k+2},  \dotsc,a_{n}) \in Z  \,\}$, by admissibility, it follows that $Y\in \mathcal{A}_{k-1}$, showing that $\square(k-1)$ holds.
\end{proof}

\begin{remark}
A result stronger than Proposition~\ref{1510a} can be proven. For example, assuming $\{ \mathcal{A}_n  \}_{n\geq 1}$ is admissible, it can be shown that $ \{\,(a_3, a_5,a_2) \mid (a_1, a_2,a_3,a_4,a_5)\in X \,\} \in \mathcal{A}_3$ for every $a_1,a_4\in \omega$ and $X\in \mathcal{A}_5$. It is not done because it is unnecessary and notationally cumbersome.
\end{remark}

\begin{definition}
Suppose $\mathfrak{A}=\{ \mathcal{A}_n  \}_{n\geq 1}$ is admissible. An \emph{$\mathfrak{A}$-ultrafilter} is an ultrafilter $U$ on $\mathcal{A}_1$ such that for every $n \geq 1$ and $X\in \mathcal{A}_{n+1} $,
$$\{ \,( a_1,\dotsc, a_{n}) \mid \{\, a_{n+1} \mid  (a_1, \dotsc, a_{n+1}) \in X \,\}\in U\,\}\in \mathcal{A}_n.$$
The set of all $\mathfrak{A}$-ultrafilters is denoted by $\beta\mathfrak{A}$.
\end{definition}

Note that by Proposition~\ref{1510a}, $\{\,a_{n+1} \mid (a_1, \dotsc,a_{n+1})\in X \,\} \in \mathcal{A}_1$ whenever $X\in \mathcal{A}_{n+1} $ and $ a_1, \dotsc, a_{n}\in \omega$.

\begin{example}\label{1912a}
Suppose $\mathcal{A}_n$ is the set of finite and cofinite subsets of $\omega^n$ for each $n \in \omega$. Then $\mathfrak{A}=\{ \mathcal{A}_n  \}_{n\geq 1}$ is admissible. All principal ultrafilters and the unique nonprincipal ultrafilter on $\mathcal{A}_1$ can easily be verified to be $\mathfrak{A}$-ultrafilters. 
\end{example}

\begin{proposition}\label{2312b}
Suppose $\mathfrak{A}=\{\mathcal{A}_{n}\}_{n \geq 1}$ is admissible. Then every principal ultrafilter on $\mathcal{A}_1$ is an $\mathfrak{A}$-ultrafilter.
\end{proposition}

\begin{proof}
Suppose $U$ is a principal ultrafilter on $\mathcal{A}_1$ generated by some $c\in \omega$. 
Suppose $n \geq 1 $ and $X \in \mathcal{A}_{n+1}$. 
For every $a_1, \dotsc, a_{n} \in \omega$, by the definition of a principal ultrafilter, $\{\, a_{n+1} \mid  (a_1, \dotsc, a_{n+1}) \in X \,\}\in U$ if and only if
$(a_1, \dotsc, a_{n}, c) \in X$. Hence, to see that $U$ is an $\mathfrak{A}$-ultrafilter, it suffices to show that
$$ Y:=\{\,(a_1, \dotsc,a_n) \mid (a_1, \dotsc, a_{n}, c) \in X \,\} \in \mathcal{A}_n.$$
Using admissibility repeatedly, $Z:=\{\,(c,a_1, \dotsc,a_n) \mid (a_1, \dotsc, a_{n}, c) \in X \,\} \in \mathcal{A}_{n+1}$. 
Hence, $Y$ is in $\mathcal{A}_n$ because $Y$ is equal to $\{\,(a_1, \dotsc,a_n) \mid (c,a_1, \dotsc, a_{n}) \in Z \,\}$.
\end{proof}

Theorem~\ref{0915a} justifies the following definition.

\begin{definition}
Suppose $\mathfrak{A}=\{ \mathcal{A}_n  \}_{n\geq 1}$ is admissible and  $U_i$ is an $\mathfrak{A}$-ultrafilter for every $i\geq 1$. Define $(U_1)_{\mathfrak{A}}$ to be $U_1$.
For each $n \geq 2$,  define  $(U_1 \otimes \dotsb \otimes U_n)_{\mathfrak{A}}$ inductively to be the ultrafilter $(U_1 \otimes \dotsb   \otimes U_{n-1})_{\mathfrak{A}} \times_{\mathcal{A}_n  }    U_n $ on $\mathcal{A}_n$.
\end{definition}

For simplicity, we will write $U_1 \otimes \dotsb \otimes U_n$ for $(U_1 \otimes \dotsb \otimes U_n)_{\mathfrak{A}}$ and $U^n$ for $(\underbrace{U \otimes \dotsb \otimes U}_{n \text{ times}})_{\mathfrak{A}}$
 whenever $\mathfrak{A}$ is understood.

To be clear, $U_1 \otimes U_2= \{\, X\in \mathcal{A}_2 \mid \{ \, a_1 \mid \{\, a_2 \mid  (a_1, a_2) \in X \,\}\in U_2\,\}
\in U_1  \,\}$. Assuming $U_1 \otimes \dotsb   \otimes U_{n-1}$ is an ultrafilter on $\mathcal{A}_{n-1}$, let us justify that 
$U_1 \otimes \dotsb   \otimes U_n$ is an ultrafilter on $\mathcal{A}_n$.
We apply Theorem~\ref{0915a} with $\mathcal{A}=\mathcal{A}_{n-1}$, $S=\omega^{n-1}$, $\mathcal{B}=\mathcal{A}_1$, $T=\omega$, 
$U=U_1 \otimes \dotsb   \otimes U_{n-1}$, $V=U_n$ and $\mathcal{C}=\mathcal{A}_n$.
Recall $((a_1, \dotsc, a_{n-1}), a_n)$ is the same as  $(a_1, \dotsc, a_n)$. Hence, the two conditions in Theorem~\ref{0915a} translate into:
\begin{enumerate}
\item $\{\,a_n \mid (a_1, \dotsc,a_{n})\in X \,\} \in \mathcal{A}_1$ whenever $X\in \mathcal{A}_n $ and $ a_1, \dotsc, a_{n-1}\in \omega$;
\item $\{ \,( a_1,\dotsc, a_{n-1}) \mid \{\, a_{n} \mid  (a_1, \dotsc, a_{n}) \in X \,\}\in U\,\}\in \mathcal{A}_{n-1}$ whenever $X\in \mathcal{A}_n$.
\end{enumerate}
By Proposition~\ref{1510a}, the first condition is satisfied  because $\mathfrak{A}$ is admissible. The second condition is satisfied because $U$ is an $\mathfrak{A}$-ultrafilter.
Therefore, $\{\, X\in \mathcal{A}_n \mid \{\, ( a_1,\dotsc, a_{n-1}) \mid \{\, a_n \mid  (a_1, \dotsc, a_n) \in X\, \}\in U_n\,\}\in U_1 \otimes \dotsb   \otimes U_{n-1}\, \}$
is an ultrafilter on $\mathcal{A}_n$, and that is our $U_1 \otimes \dotsb \otimes U_n$.

\begin{proposition}\label{0913a}
Suppose $\mathfrak{A}=\{ \mathcal{A}_n  \}_{n \geq 1}$ is admissible and $U_1, \dotsc,U_n$ are $\mathfrak{A}$-ultrafilters, where $n \geq 2$.  Then for every $X \in \mathcal{A}_n$ and  $1 \leq k \leq n-1$, 
$$W:=\{(a_1, \dotsc,a_{n-k}) \mid \{(a_{n-k+1}, \dotsc,a_n) \mid (a_1, \dotsc,a_n)\in X \} \in U_{n-k+1} \otimes \dotsb \otimes U_n\}$$
is in $\mathcal{A}_{n-k}$ and  $X \in U_1 \otimes \dotsb \otimes U_n$ if and only if $W \in U_1 \otimes \dotsb \otimes U_{n-k}$.
\end{proposition}
(Note that by Proposition~\ref{1510a}, $ \{\,(a_{n-k+1}, \dotsc,a_n) \mid (a_1, \dotsc,a_n)\in X \,\} \in \mathcal{A}_k$ whenever $a_1, \dotsc, a_{n-k} \in \omega$.)

\begin{proof}
We proceed by induction on $n$. The base step $n=2$ is immediate by the definition of $\mathfrak{A}$-ultrafilter and $U_1 \otimes U_2$. Assume $n>2$.
Suppose $X \in \mathcal{A}_n$, $1 \leq k \leq n-1$ and $W$ is as stated in the lemma. If $k=1$, then 
the conclusion follows similarly. Now, assume $k>2$.
Since $U_n$ is an $\mathfrak{A}$-ultrafilter, $Z:=\{\,(a_1, \dotsc,a_{n-1}) \mid \{\,a_n\mid (a_1, \dotsc,a_n)\in X \,\} \in U_n\,\}$ is in $\mathcal{A}_{n-1}$
and $X \in U_1 \otimes \dotsb \otimes U_n$ if and only if $Z \in U_1 \otimes \dotsb \otimes U_{n-1}$.
Now,  $W$ is equal to  $\{(a_1, \dotsc,a_{n-k}) \mid \{\,(a_{n-k+1}, \dotsc,a_{n-1})\mid \{\,a_n \mid (a_1, \dotsc,a_n)\in X \,\} \in U_n \,\}\in U_{n-k+1} \otimes \dotsb \otimes U_{n-1}\,\}$, which in turn is equal to
 $\{\,(a_1, \dotsc,a_{n-k}) \mid \{\,(a_{n-k+1},\dotsc a_{n-1}) \mid (a_1, \dotsc,a_{n-1})\in
Z\,\} \in U_{n-k+1} \otimes \dotsb \otimes U_{n-1}\,\}$.
By the induction hypothesis, $W$ is in $\mathcal{A}_{n-k}$ and
$Z \in U_1 \otimes \dotsb \otimes U_{n-1}$ if and only if $W \in U_1 \otimes \dotsb \otimes U_{n-k}$.
Therefore, $X \in U_1 \otimes \dotsb \otimes U_n$ if and only if $W \in U_1 \otimes \dotsb \otimes U_{n-k}$.
\end{proof}

\begin{definition}\label{1510b}
Suppose $\mathfrak{A}=\{ \mathcal{A}_n  \}_{n\geq 1}$ is admissible.
An $m$-ary operation $f$ on $\omega$ is an \emph{$\mathfrak{A}$-operation} if{f}
for every $n \geq 1$ and $X \in\mathcal{A}_n$,
$$\{\,( a_1, \dotsc, a_{n+m-1} ) \mid ( a_1,\dotsc, a_{n-1}, f(a_{n}, \dotsc,,a_{n+m-1}) ) \in X\,\} \in \mathcal{A}_{n+m-1}.$$ 
In particular,  $f^{-1}[X] \in \mathcal{A}_m$ for all $X \in\mathcal{A}_1$.
\end{definition}


\begin{theorem}\label{0911a}
Suppose $\mathfrak{A}=\{ \mathcal{A}_n  \}_{n\geq 1}$ is admissible.
Suppose $f$ is an $m$-ary $\mathfrak{A}$-operation. If $U_1, \dotsc, U_m$ are $\mathfrak{A}$-ultrafilters, then
the collection $\{\, X\in \mathcal{A}_1  \mid f^{-1}[X]\in U_1\otimes \dotsb\otimes U_m\,\}$ is an $\mathfrak{A}$-ultrafilter.
\end{theorem}

\begin{proof}
Let $\mathscr{C}$ be $\{\, X\in \mathcal{A}_1  \mid f^{-1}[X]\in U_1\otimes \dotsb\otimes U_m\,\}$. Since  $f^{-1}[X] \in \mathcal{A}_m$ for all $X \in\mathcal{A}_1$ and  
$U_1\otimes \dotsb\otimes U_m$ is an ultrafilter on $\mathcal{A}_m$, by Proposition~\ref{1006a}, $\mathscr{C}$ is an ultrafilter on $\mathcal{A}_1$.
Fix $n \geq 1$ and $X \in \mathcal{A}_{n+1}$. 
We need to show that $W:=\{ \,( a_1,\dotsc, a_{n}) \mid \{\, a_{n+1} \mid  (a_1, \dotsc, a_{n+1}) \in X \,\}\in \mathscr{C}\,\}\in \mathcal{A}_{n}$.
For every $a_1, \dotsc, a_{n}\in \omega$, it is easy to verify that $\{\, a_{n+1} \mid  (a_1, \dotsc, a_{n+1}) \in X \,\}\in \mathscr{C}$ if and only if $\{\, (a_{n+1}, \dotsc,a_{n+m}) \mid (a_1, \dotsc, a_{n}, f(a_{n+1}, \dotsc,a_{n+m}) ) 
\in X\,\} \in U_1\otimes \dotsb\otimes U_{m}
$. 

Since $f$ is an $\mathfrak{A}$-operation, there exists $Z \in \mathcal{A}_{n+m}$ such that  $(a_1, \dotsc,a_{n+m}) \in Z$ if and only if
$(a_1, \dotsc, a_{n}, f(a_{n+1}, \dotsc,a_{n+m}) ) \in X$.
Therefore, 
$W$ is equal to $\{ \,( a_1,\dotsc, a_{n}) \mid \{\, (a_{n+1}, \dotsc,a_{n+m}) \mid (a_1, \dotsc,a_{n+m}) \in Z\,\} \in U_1\otimes \dotsb\otimes U_m\,\}$.
By Proposition~\ref{0913a}, $W$ is in $\mathcal{A}_{n}$.
\end{proof}

Theorem~\ref{0911a} justifies the following definition.

\begin{definition}
Suppose $\mathfrak{A}=\{ \mathcal{A}_n  \}_{n\geq 1}$ is admissible and $f$  is an $m$-ary $\mathfrak{A}$-operation. Define $f_*^{\mathfrak{A}}$ to be the $m$-ary operation on $\beta\mathfrak{A}$ given by
$$f_*^{\mathfrak{A}}(U_1, \dotsc, U_m)=\{\, X\in \mathcal{A}_1  \mid f^{-1}[X]\in U_1\otimes \dotsb\otimes U_m\,\}.$$ 
\end{definition}

In fact, $f_*^{\mathfrak{A}}$ is an operation on $\beta\mathfrak{A}$ that extends $f$, when we identify the natural numbers with the principal ultrafilters on $\mathcal{A}_1$.

\begin{example}
Let $\mathfrak{A}$ be $\{\mathcal{P}(\omega^n)\}_{n\geq 1}$. Then $\mathfrak{A}$ is trivially admissible. If $f$ is any operation on $\omega$, then $f$ is immediately
an $\mathfrak{A}$-operation. In fact, $(\beta\mathfrak{A},f_*^{\mathfrak{A}} )$ is equal to $(\beta\omega, f_*) $.
\end{example}


\begin{proposition}\label{3112a}
Suppose $\mathfrak{A}=\{ \mathcal{A}_n  \}_{n\geq 1}$ and $\mathfrak{B}=\{ \mathcal{B}_n  \}_{n\geq 1}$ are admissible.
For every $n \geq 1$, assume $\mathcal{A}_n\subseteq \mathcal{B}_n$, $U_n$ is an $\mathfrak{A}$-ultrafilter and $V_n$ is a $\mathcal{B}$-ultrafilter, where $U_n=V_n\cap \mathcal{A}_1$. Then for every $m \geq 1$, we have
$$(U_1\otimes \dotsb \otimes U_m)_{\mathfrak{A}} = (V_1\otimes \dotsb \otimes  V_m)_{\mathfrak{B}} \cap \mathcal{A}_m.$$
Furthermore, if $f$ is an $m$-ary $\mathfrak{A}$-operation as well as a $\mathfrak{B}$-operation, then  $f_*^{\mathfrak{A}}(U_1, \dotsc, U_m )= f_*^{\mathfrak{B}}(V_1, \dotsc, V_m )\cap \mathcal{A}_1$.
\end{proposition}

\begin{proof}
We prove the first part by induction on $m$. The base step holds trivially. For the inductive step, first, suppose $X\in \mathcal{A}_{m+1}$.
Let $W_X^{\mathfrak{A}}$ denote $\{\, (a_1, \dotsc, a_m) \mid \{\, a_{m+1} \mid (a_1, \dotsc, a_{m+1})\in X \,\}\in U_{m+1}\,\}$
and $W_X^{\mathfrak{B}}$ denote $\{\, (a_1, \dotsc, a_m) \mid \{\, a_{m+1} \mid (a_1, \dotsc, a_{m+1})\in X \,\}\in V_{m+1}\,\}$.
Note that $W_X^{\mathfrak{A}}=W_X^{\mathfrak{B}}$ because $U_{m+1}=V_{m+1}\cap \mathcal{A}_1$. Additionally, since $U_{m+1}$ is an $\mathfrak{A}$-ultrafilter, $W_X^{\mathfrak{A}}\in \mathcal{A}_m$. Therefore,
\begin{align*}
&(V_1\otimes \dotsb \otimes  V_{m+1})_{\mathfrak{B}} \cap \mathcal{A}_{m+1}\\
&=\{ X\in \mathcal{B}_{m+1} \mid W_X^{\mathfrak{B}} \in (V_1\otimes \dotsb \otimes  V_{m})_{\mathfrak{B}} \}\cap \mathcal{A}_{m+1}\\
&=\{\, X\in \mathcal{A}_{m+1} \mid W_X^{\mathfrak{A}}
\in (V_1\otimes \dotsb \otimes  V_{m})_{\mathfrak{B}} \,\}\\
&=\{\, X\in \mathcal{A}_{m+1} \mid W_X^{\mathfrak{A}}
\in (U_1\otimes \dotsb \otimes  U_{m})_{\mathfrak{A}} \,\} \qquad \text{(by the induction hypothesis)}\\
&=(U_1\otimes \dotsb \otimes U_{m+1})_{\mathfrak{A}}
\end{align*}
For the second part, 
$f_*^{\mathfrak{B}}(V_1, \dotsc, V_m )\cap \mathcal{A}_1=\{\, X\in \mathcal{A}_1\mid f^{-1}[X]\in (V_1\otimes \dotsb \otimes  V_m)_{\mathfrak{B}}\,\}=\{\, X\in \mathcal{A}_1\mid f^{-1}[X]\in (U_1\otimes \dotsb \otimes  U_m)_{\mathfrak{A}}\,\}= 
f_*^{\mathfrak{A}}(U_1, \dotsc, U_m )$. The second equality holds because of the first part and $f^{-1}[X]\in \mathcal{A}_m$ for all $X\in \mathcal{A}_1$.
\end{proof}


\begin{theorem}\label{2312a}
Suppose $\mathfrak{A}=\{ \mathcal{A}_n  \}_{n\geq 1}$ is admissible. If $(\omega,f)$ is a semigroup and $f$ is an $\mathfrak{A}$-operation, then $(\beta\mathfrak{A}, f_\ast^{\mathfrak{A}})$ is a semigroup. 
\end{theorem}

\begin{proof}
Let   $\mathfrak{B}$ equal to  $\{\mathcal{P}(\omega^n)\}_{n\geq 1}$. 
Then $(\beta\mathfrak{B}, f_\ast^{\mathfrak{B}})$, being equal to $(\beta\omega, f_*)$, is a semigroup. Suppose $U_1, U_2,U_3\in \beta\mathfrak{A}$. Let $V_1, V_2,V_3$ be any ultrafilter on $\mathcal{P}(\omega)$ that extends $U_1,U_2,U_3$ respectively. 
By Proposition~\ref{3112a}, $f_*^{\mathfrak{A}}(U_1, U_2)=f_*^{\mathfrak{B}}(V_1, V_2)\cap \mathcal{A}_1$
and $f_*^{\mathfrak{A}}(U_2, U_3)=f_*^{\mathfrak{B}}(V_2, V_3)\cap\mathcal{A}_1$.
Therefore,
$f_*^{\mathfrak{A}}(f_*^{\mathfrak{A}}(U_1, U_2),U_3) = f_*^{\mathfrak{B}}(f_*^{\mathfrak{B}}(V_1, V_2),V_3)\cap \mathcal{A}_1=f_*^{\mathfrak{B}}  (V_1,f_*^{\mathfrak{B}}(V_2, V_3) ) \cap \mathcal{A}_1 =f_*^{\mathfrak{A}}  (U_1,f_*^{\mathfrak{A}}(U_2, U_3)) $.
The first and the third equality follow from Proposition~\ref{3112a}.
\end{proof}

\begin{definition}\label{1120a}
Suppose $\mathfrak{A}=\{ \mathcal{A}_n  \}_{n\geq 1}$ is admissible and $U$ is an $\mathfrak{A}$-ultrafilter. Suppose $\mathcal{F}$ is a collection of  $\mathfrak{A}$-operations. We say that $U$ is \emph{idempotent} for $\mathcal{F}$ with respect to $\mathfrak{A}$ if{f} $f_*^{\mathfrak{A}}(U,\dotsc,U)=U$ whenever $f \in \mathcal{F}$. 
\end{definition}


\begin{example}\label{2706e}
Suppose $\mathcal{A}_n$ is the set of finite and cofinite subsets of $\omega^n$ for each $n \in \omega$. Suppose $f$ is any binary operation on $\omega$ such that $f^{-1}[\{a\}]$ is finite for all $a\in \omega$. Then $f$ is an $\mathfrak{A}$-operation. Furthermore, the unique nonprincipal $\mathfrak{A}$-ultrafilter is idempotent for $f$ (or $\{f\}$ precisely) with respect to $\mathfrak{A}$. However, $(\omega,f)$ need not be a Ramsey algebra and ultrafilters on $\mathcal{P}(\omega)$ idempotent for $f$ may not exist.
\end{example}

Suppose $\mathfrak{A}=\{ \mathcal{A}_n  \}_{n\geq 1}$ is admissible and $f$ is a binary associative $\mathfrak{A}$-operation.  
Take any ultrafilter $V$ on $\omega$ such that $f_*(V,V)=V$. Let $U$ be $V\cap \mathcal{A}_1$.
If $U$ is an $\mathfrak{A}$-ultrafilter, by Proposition~\ref{3112a}, $U$ is idempotent for $f$ with respect to $\mathfrak{A}$.
Equivalently, if an $\mathfrak{A}$-ultrafilter $U$ can be extended to an ultrafilter on $\omega$ idempotent for $f$, then $U$ is idempotent for $f$ with respect to $\mathfrak{A}$.
However, Example~\ref{2706e} suggests that this may not be possible, even if $U$ is idempotent for $f$ with respect to $\mathfrak{A}$.  
Therefore, although $(\beta\mathfrak{A}, f_\ast^{\mathfrak{A}})$ is a semigroup, it is conceivable that an $\mathfrak{A}$-ultrafilter idempotent for $f$ with respect to $\mathfrak{A}$ may not exist. If that is true, one can further analyze the admissible $\mathfrak{A}$ such that the existence of such idempotent ultrafilters is guaranteed. 

Nevertheless, our framework is sufficient and minimal in some sense to allow the construction of a homegeneous sequence for every set in an idempotent ultrafilter.

\begin{theorem}
Suppose $\mathfrak{A}=\{ \mathcal{A}_n  \}_{n\geq 1}$ is admissible and $f$ is a binary associative $\mathfrak{A}$-operation. If $U$ is an  $\mathfrak{A}$-ultrafilter idempotent for $f$ with respect to $\mathfrak{A}$, then for every $X\in U$, there exists $\vec{a}\in {^\omega}\omega$ such that $\FR_{\mathcal{F}}(\vec{a})\subseteq X$.
\end{theorem}

\begin{proof}
The standard Galvin's method of constructing homogeneous sequences, for example, as in Theorem~5.8 of \cite{HS12} can be carried out. 
\end{proof}

\section{Existence of Idempotent Ultrafilters}



In this section, we will exbibit the existence of idempotent ultrafilters for every nondegenerate Ramsey algebra in some countable sense.
To do this, countability is imposed on the general framework introduced in the previous section. 
Apart from the reason from the point of reverse mathematics, the countability assumption is actually crucial in our construction.
We start off with some lemmas.

\begin{lemma}\label{1120b}
Suppose $\mathcal{F}$ is a countable collection of operations on $\omega$. 
Suppose $S_n$ is a countable collection of subsets of $\omega^n$ for every $n \geq 1$.
Then there exists an admissible $\mathfrak{A}=\{ \mathcal{A}_n  \}_{n\geq 1}$ such that
$\mathcal{F}$ is a collection of  $\mathfrak{A}$-operations and
$\mathcal{A}_n$ is a countable superset of $S_n$ for every $n\geq 1$.
\end{lemma}

\begin{proof}
For each $n\geq 1$, let $\cap_n$, $\cup_n$ and  $\comp_n$ denote the respective operations on $\mathcal{P}(\omega^n)$ defined by $\cap_n(X,Y)=X\cap Y$, $\cup_n(X,Y)=X\cup Y$, $\comp_n(X)=\omega^n\backslash X$. 
For $n \geq 2$, the operation $\cyc_n$ on $\mathcal{P}(\omega^n)$ is defined by $\cyc_n(X)= \{\,(a_2, \dotsc,a_n, a_1) \mid (a_1, \dotsc,a_n)\in X \,\}$.
For each $n\geq 2$ and $c\in \omega$, let $\fib_n^c \colon \mathcal{P}(\omega^{n}) \rightarrow \mathcal{P}(\omega^{n-1})$ be defined by 
$\fib_n^c(X)= \{\,(a_1, \dotsc,a_{n-1}) \mid (c,a_1, \dotsc,a_{n-1})\in X \,\}$.
For every $f\in \mathcal{F}$ and $n\geq 1$, say $f$ is $m$-ary, let $\pre_n^f \colon \mathcal{P}(\omega^n) \rightarrow \mathcal{P}(\omega^{n+m-1})$ be defined by 
$\pre_n^f(X)= \{\,( a_1, \dotsc, a_{n+m-1} ) \mid ( a_1,\dotsc, a_{n-1}, f(a_{n}, \dotsc,,a_{n+m-1}) ) \in X\,\}$.
Suppose $\mathscr{H}=\{\, \cup_n, \cap_n, \comp_n \mid n\geq 1\,\} \bigcup \{\, \cyc_n, \fib_n^c \mid n\geq 2, c\in \omega\,\} \bigcup \{\, \pre_n^f \mid f\in \mathcal{F}, n\geq 1\,\} $. 
Then $\mathscr{H}$ is a countable collection of operations on
$\{ \mathcal{P}(\omega^n) \}_{n\geq 1}$. By Theorem~\ref{2111}, choose $\mathfrak{A}=\{ \mathcal{A}_n  \}_{n\geq 1}$ such that $\mathcal{A}_n$ is a countable superset of $S_n$
for every $n\geq 1$ and that $\mathfrak{A}$ is closed under $\mathscr{H}$. For every $n \geq 1$, since $\mathfrak{A}$ is  closed under $\{\cap_n,\cup_n,\comp_n \mid n \geq 1\}$, $\mathcal{A}_n$ is a field of sets over $\omega^n$. Since $\mathfrak{A}$ is closed under $\{ \,\cyc_n, \adm_n^c \mid n \geq 2, c\in \omega\,\}$, it is admissible.
For every $f\in \mathcal{F}$, since $\mathfrak{A}$ is closed under $\{\, \pre_n^f \mid n\geq 1\,\} $, it follows that $f$ is an $\mathfrak{A}$-operation.
\end{proof}

\begin{lemma}\label{2311a}
Suppose $\mathfrak{A}_i=\{\mathcal{A}_{n,i}\}_{n \geq 1}$ is admissible for every $i$ in an index set $I$. Suppose  $\mathcal{F}$ is a collection of operations on $\omega$, each of which is an $\mathfrak{A}_i$-operation for every $i\in I$. For every $n \geq 1$, take $\mathcal{A}_n= \bigcup_{i \in I} \mathcal{A}_{n,i}$. Then $\mathfrak{A}= \{\mathcal{A}_n\}_{n\geq 1}$ is admissible and $\mathcal{F}$ is a collection of $\mathfrak{A}$-operations.
\end{lemma}

\begin{proof}
Straightforward.
\end{proof}


Definition~\ref{0207a} is modified very slightly to the following. It serves us well as shown by the following theorem.

\begin{definition}\label{1119e}
Suppose $\mathcal{F}$ is a collection of operations on $\omega$ and $\mathcal{A}$ is a field of sets over $\omega$. An ultrafilter $U$ on $\mathcal{A}$ is \emph{strongly reductible} for $\mathcal{F}$ if{f}  for every  $X \in U$, there exists $\vec{a} \in {^\omega}\omega$ such that 
$\FR_{\mathcal{F}}(\vec{a}) \subseteq X$ and $\FR_{\mathcal{F}}(\vec{a}-i) \in U$ for all $i \in \omega$.
\end{definition}

\begin{theorem}\label{1207f}
Suppose $\mathfrak{A}=\{ \mathcal{A}_n  \}_{n\geq 1}$ is admissible and $U$ is an $\mathfrak{A}$-ultrafilter. Suppose $\mathcal{F}$ is a collection of $\mathfrak{A}$-operations and
$U$ is strongly reductible for $\mathcal{F}$. Then $U$ is idempotent for $\mathcal{F}$ with respect to $\mathfrak{A}$.
\end{theorem}

\begin{proof}
Suppose  $f$ is an $m$-ary  operation in $\mathcal{F}$. Since $f^{\mathfrak{A}}_*(U^m)$ is an ultrafilter on $\mathcal{A}_1$, to see that $U$ is idempotent for $f$ with respect to $\mathfrak{A}$, it suffices to show that $X\in f^{\mathfrak{A}}_*(U^m)$ for every $X\in U$. Suppose $X\in U$. Choose $\vec{b} \in {^\omega}\omega$ such that 
$\FR_{\mathcal{F}}(\vec{b}) \subseteq X$ and $\FR_{\mathcal{F}}(\vec{b}-i) \in U$ for all $i \in \omega$. We need to show that $f^{-1}[X] \in U^m$ and this follows from the case $k=m$ in the following claim.

\begin{claim}
For each $1 \leq k \leq m$ and $a_1, \dotsc, a_{m-k}\in \omega$ such that $\langle a_1, \dotsc, a_{m-k}\rangle\unlhd_{\mathcal{F}} \vec{b}\!\upharpoonright\! N$ for some $N\in\omega$, it follows that
$\{\, (a_{m-k+1}, \dotsc, a_m) \mid f(a_1, \dotsc, a_m) \in X \,\}\in U^{k}$.
\end{claim}

The claim is proved by induction on $k$.  First, note that $\{\, (a_{m-k+1}, \dotsc, a_m) \mid f(a_1, \dotsc, a_m) \in X \,\}\in \mathcal{A}_k$ for every $1 \leq k \leq m$ and $a_1, \dotsc, a_{m-k}\in \omega$ 
 by $f^{-1}[X]\in \mathcal{A}_m$ and Proposition \ref{1510a}.
Consider the base case $k=1$. Suppose $\langle a_1, \dotsc, a_{m-1}\rangle\unlhd_{\mathcal{F}} \vec{b} \!\upharpoonright \! N$ for some $N \in \omega$. By the definition of orderly term, $f(a_1, \dotsc, a_m) \in \FR_{\mathcal{F}}(\vec{b})$ for all $a_m\in \FR_{\mathcal{F}}(\vec{b}-N)$.
 Thus $\FR_{\mathcal{F}}(\vec{b}-N)\subseteq \{\, a_m \mid f(a_1, \dotsc, a_m) \in X \,\}$. Since $\FR_{\mathcal{F}}(\vec{b}-N)$ is in $ U$, so is $\{\, a_m \mid f(a_1, \dotsc, a_m) \in X \,\}$. Now, assume $1<k \leq m$. 
Suppose $\langle a_1, \dotsc, a_{m-k}\rangle\unlhd_{\mathcal{F}} \vec{b} \!\upharpoonright \! N$ for some $N \in \omega$.
By Lemma \ref{0913a}, it suffices to show that $\{\, a_{m-k+1} \mid \{\, (a_{m-k+2}, \dotsc, a_m) \mid f(a_1, \dotsc, a_m) \in X \,\}\in U^{k-1} \,\}\in U$. This will follow if  $\FR_{\mathcal{F}}(\vec{b}-N) \subseteq \{\, a_{m-k+1} \mid \{\, (a_{m-k+2}, \dotsc, a_m) \mid f(a_1, \dotsc, a_m) \in X \,\}\in U^{k-1} \,\}$. Suppose $a_{m-k+1} \in \FR_{\mathcal{F}}(\vec{b}-N)$. Then $\langle a_1, \dotsc, a_{m-k+1}\rangle$ is a reduction of some initial segment of $\vec{b}$  with respect to $\mathcal{F}$. By the induction hypothesis, $\{\, (a_{m-k+2}, \dotsc, a_m) \mid f(a_1, \dotsc, a_m) \in X \,\}\in U^{k-1}$.
\end{proof}


In the case of addition, the next theorem is known as the Iterated Hindman Theorem, which is equivalent to Hindman's Theorem in $\mathsf{ZFC}$. Before that, we state a technical crucial lemma.

\begin{lemma}\cite{wcT13b} \label{0729}
Suppose $\langle \vec{a}_n \rangle_{n \in \omega}$ is a sequence in ${^\omega}\!\omega$  such that  
 $\vec{a}_{n+1} \leq_{\mathcal{F}} \vec{a}_n$ for all $n \in \omega$. Then there exists $\vec{b}\in  {^\omega}\!\omega$ such that $\vec{b}-n \leq_{\mathcal{F}} \vec{a}_n$ for all $n \in \omega$.
\end{lemma}

\begin{theorem}\label{1211a}
Suppose $(\omega, \mathcal{F})$ is a Ramsey algebra. Then
for every $\vec{a} \in {^\omega}\omega$ and every indexed collection  $\{A_n\}_{n \in \omega}$ of subsets of $\omega$, there exists $\vec{b} \leq_{\mathcal{F}}\vec{a}$ such that $\FR_{\mathcal{F}}(\vec{b}-n)$ is either contained in or disjoint from $A_n$ whenever  $n \in \omega$. 
\end{theorem}

\begin{proof}
Suppose $\vec{a} \in {^\omega}\omega$ and $\{A_n\}_{n \in \omega}$ is an indexed collection of subsets of $\omega$. Since the algebra is Ramsey, for every $n \in \omega$, we can choose $\vec{a}_n \in {^\omega}\omega$ inductively with $\vec{a}_0 \leq_{\mathcal{F}} \vec{a}$ such that $\vec{a}_{n+1} \leq_{\mathcal{F}} \vec{a}_n$ and $\FR_{\mathcal{F}}(\vec{a}_n)$ is either contained in or disjoint from $A_n$ for each $n \in \omega$. By Lemma \ref{0729}, choose $\vec{b}\in  {^\omega}\omega$ such that $\vec{b}-n \leq_{\mathcal{F}} \vec{a}_n$ for all $n \in \omega$. 
By transitivity, $\vec{b} =\vec{b}-0 \leq_{\mathcal{F}} \vec{a}$. The conclusion follows since $\FR_{\mathcal{F}}(\vec{b}-n) \subseteq \FR_{\mathcal{F}}(\vec{a}_n)$ for all $n \in \omega$.
\end{proof}

\begin{lemma}\label{0805a}
Suppose $(\omega, \mathcal{F})$ is a Ramsey algebra and  $\mathcal{A}$ is a countable field of sets over $\omega$. For every $\vec{a} \in {^\omega}\omega$, there exists $\vec{b} \leq_{\mathcal{F}} \vec{a}$ such that 
$\{\,X \in \mathcal{A} \mid \FR_{\mathcal{F}}(\vec{b}-n) \subseteq X \text{ \textnormal{for some} } n \in \omega\,\}$
 is an ultrafilter on $\mathcal{A}$.
\end{lemma}

\begin{proof}
Fix an enumeration $A_0, A_1, A_2, \dotsc$ of the sets in $\mathcal{A}$. Suppose $\vec{a} \in {^\omega}\omega$. By Theorem~\ref{1211a}, choose $\vec{b} \leq_{\mathcal{F}}\vec{a}$ such that $\FR_{\mathcal{F}}(\vec{b}-n)$ is either contained in or disjoint from $A_n$ for each  $n \in \omega$. 
We claim that  $U:=\{\,X \in \mathcal{A} \mid \FR_{\mathcal{F}}(\vec{b}-n) \subseteq X \text{ for some } n \in \omega\,\}$
is an ultrafilter on $\mathcal{A}$. Clearly, $\omega\in U$ and $\emptyset \notin U$. Suppose $A, B \in U$, say
$\FR_{\mathcal{F}}(\vec{b}-n) \subseteq A$ and $\FR_{\mathcal{F}}(\vec{b}-m) \subseteq B$  for some  $n,m \in \omega$. Then $A\cap B \in \mathcal{A}$ and 
$\FR_{\mathcal{F}}(\vec{b}-\max\{n,m\}) \subseteq A\cap B$, implying that $A\cap B\in U$. Meanwhile, by the choice of $\vec{b}$, either $A_n\in U$ or ${A_n}^{\!\!c} \in U$ for each $n \in \omega$. Hence, $U$ is an ultrafilter on $\mathcal{A}$.
\end{proof}

\begin{theorem}\label{0905b}
Assume $(\omega, \mathcal{F})$ is a nondegenerate Ramsey algebra and $\mathcal{F}$ is countable. Suppose $S_n$ is a countable collection of subsets of $\omega^n$ for every $n \geq 1$. Then there exists an admissible $\mathfrak{A}=\{ \mathcal{A}_n  \}_{n\geq 1}$ such that
\begin{enumerate}
\item $\mathcal{A}_n$ is countable superset of $S_n$ for every $n \geq 1$;
\item $\mathcal{F}$ is a collection of $\mathfrak{A}$-operations;
\item there exists a nonprincipal $\mathfrak{A}$-ultrafilter $U$ strongly reductible for $\mathcal{F}$. 
\end{enumerate}
\end{theorem}

\begin{proof}
First of all, by Theorem~\ref{130613b}, fix a sequence $\vec{a}$ such that $\FR_{\mathcal{F}}(\vec{b})$ is infinite whenever $\vec{b}\leq_{\mathcal{F}} \vec{a}$.
For every $m \in \omega$, we will construct inductively an admissible $\mathfrak{A}_m=\{\mathcal{A}_{n,m}\}_{n \geq 1}$, an ultrafilter $U_m$ on $\mathcal{A}_{1,m}$ and a sequence $\vec{a}_m$ of natural numbers. By Lemma~\ref{1120b}, choose an admissible $\mathfrak{A}_0=\{\mathcal{A}_{n,0}\}_{n \geq 1}$ such that $\mathcal{A}_{n,0}$ is countable superset of $S_n$ for every $n \geq 1$ and $\mathcal{F}$ is a collection of $\mathfrak{A}_0$-operations. Since $(\omega, \mathcal{F})$ is a  Ramsey algebra, by Lemma~\ref{0805a}, choose $\vec{a}_0\leq_{\mathcal{F}}\vec{a}$ such that $U_0=\{\,X \in \mathcal{A}_{1,0} \mid \FR_{\mathcal{F}}(\vec{a}_0-i) \subseteq X \text{ for some } i \in \omega\,\}$ is an ultrafilter on $\mathcal{A}_{1,0}$. At the inductive step, by Lemma~\ref{1120b}, choose an admissible $\mathfrak{A}_{m+1}=\{\mathcal{A}_{n,m+1}\}_{n \geq 1}$ such that $\mathcal{A}_{n,m+1}$ is countable superset of $\mathcal{A}_{n,m}$ for every $n \geq 1$,  that $\mathcal{F}$ is a collection of $\mathfrak{A}_{m+1}$-operations  and additionally  
\begin{itemize}
\item $\FR_{\mathcal{F}}(\vec{a}_m -i)\in \mathcal{A}_{1,m+1}$ for every $ i \in \omega$;
\item for every $n \geq 1$, we have $\{ \,( a_1,\dotsc, a_n) \mid \{\, a_{n+1} \mid  (a_1, \dotsc, a_{n+1}) \in X \,\}\in U_{m}\,\}\in \mathcal{A}_{n,m+1}$ for all $X \in \mathcal{A}_{n+1,m}$.
\end{itemize}
Now, by Lemma~\ref{0805a}, choose $\vec{a}_{m+1}\leq_{\mathcal{F}}\vec{a}_m$ such that $U_{m+1}=\{\,X \in \mathcal{A}_{1,m+1} \mid \FR_{\mathcal{F}}(\vec{a}_{m+1}-i) \subseteq X \text{ for some } i \in \omega\,\}$ is an ultrafilter on $\mathcal{A}_{1,m+1}$.

For every $n \geq 1$, take $\mathcal{A}_n= \bigcup_{m\in \omega} \mathcal{A}_{n,m}$. Take $\mathfrak{A}$ to be $\{\mathcal{A}_n\}_{n\geq 1}$.
Clearly, $\mathcal{A}_n$ is countable for every $n \geq 1$ since it is a countable union of countable sets. By our construction, $S_n\subseteq \mathcal{A}_{n,0}\subseteq  \mathcal{A}_n$ for every $n \geq 1$. 
By Lemma~\ref{2311a}, $\mathfrak{A}$ is admissible and  $\mathcal{F}$ is a collection of $\mathfrak{A}$-operations. 

Now, take $U=\bigcup_{m \in \omega} U_m$. 
We claim that $\FR_{\mathcal{F}}(\vec{a}_m-i)$ is infinite and is in $U$ for every $m,i \in \omega$.
To see this, suppose $m,i \in \omega$. By transitivity, $\vec{a}_m-i\leq_{\mathcal{F}} \vec{a}$. By our choice of $\vec{a}$, we know $\FR_{\mathcal{F}}(\vec{a}_m-i)$ is infinite.
Since $\vec{a}_{m+1}\leq_{\mathcal{F}} \vec{a}_m$, we have $\FR_{\mathcal{F}}(\vec{a}_{m+1}-i)\subseteq \FR_{\mathcal{F}}(\vec{a}_m-i)$. Thus $\FR_{\mathcal{F}}(\vec{a}_m-i) \in U_{m+1}\subseteq U$ as $\FR_{\mathcal{F}}(\vec{a}_m-i) \in  \mathcal{A}_{1,m+1}$. Therefore, $\FR_{\mathcal{F}}(\vec{a}_m-i)\in U$.

Because $U_0\subseteq U_1\subseteq U_2\, \dotsb $, it follows easily that $U$ is an ultrafilter on $\mathcal{A}_1$.
To see that $U$ is an $\mathfrak{A}$-ultrafilter, suppose $n \geq 1$ and $X\in \mathcal{A}_{n+1}$.
Then $X\in \mathcal{A}_{n+1,m}$ for some $m \in \omega$. Since $\mathfrak{A}_m$ is admissible, $\{\, a_{n+1} \mid  (a_1, \dotsc, a_{n+1}) \in X \,\}\in \mathcal{A}_{1,m}$ for  every $a_1, a_2, \dotsc, a_{n}\in \omega$. Therefore, $\{ \,( a_1,\dotsc, a_{n}) \mid \{\, a_{n+1} \mid  (a_1, \dotsc, a_{n+1}) \in X \,\}\in U\,\}
=\{ \,( a_1,\dotsc, a_{n}) \mid \{\, a_{n+1} \mid  (a_1, \dotsc, a_{n+1}) \in X \,\}\in U_m\,\}\in \mathcal{A}_{n,m+1}\subseteq \mathcal{A}_{n}$.
Finally, suppose $X\in U$. Then $X\in U_m$ for some $m\in \omega$. Hence, $\FR_{\mathcal{F}}(\vec{a}_{m}-i) \subseteq X$ for some $i \in \omega$.
By our claim, $X$ is infinite and $\FR_{\mathcal{F}}(\vec{a}_m-j)\in U$ for each $j \geq i$. Therefore,
$U$ is nonprincipal and strongly reductible.
\end{proof}

\section{Extension of Idempotentness to Orderly Terms} 


In this section, we will show that the analogue of Theorem~\ref{1007a} holds in our framework.

Suppose $\mathfrak{A}=\{ \mathcal{A}_n  \}_{n\geq 1}$ is admissible. If $X\subseteq \omega^n$ for some $n\in \omega$, whose value is implicitly implied by the context, then we write $X\overset{.}{\in} \mathfrak{A}$ to mean $X\in \mathcal{A}_n$. This is done out of convenience as seen in the following lemma.

\begin{lemma}\label{1207c}
Suppose $\mathfrak{A}=\{ \mathcal{A}_n  \}_{n\geq 1}$ is admissible.
Suppose $m \geq 1 $ and $h_1, \dotsc, h_m$ are $\mathfrak{A}$-operations. Then for every $n \geq 1$ and $X\in \mathcal{A}_{n+m-1}$,
$$ \{\, (a_1, \dotsc,a_{n-1}, \bar{x}_1, \dotsc, \bar{x}_m) \mid (a_1, \dotsc,a_{n-1}, h_1(\bar{x}_1), \dotsc, h_m(\bar{x}_m) ) \in X\,\}\overset{.}{\in} \mathfrak{A}$$
\end{lemma}

\begin{proof}
The proof proceeds by induction on $m$. The base case $m=1$ holds since $h_1$ is an
$\mathfrak{A}$-operation. Assume $m >1$. Suppose $n \geq 1$ and $X\in \mathcal{A}_{n+m-1}$.
We need to show that $W\overset{.}{\in} \mathfrak{A}$, where $W$ is as stated in the lemma.
 By the induction hypothesis,
 $ Y_1:=\{\,(a_1, \dotsc,a_{n}, \bar{x}_2, \dotsc, \bar{x}_m) \mid (a_1, \dotsc,a_{n}, h_2(\bar{x}_2), \dotsc, h_m(\bar{x}_m)) \in X\,\}\overset{.}{\in}  \mathfrak{A}$.
By applying admissibility repeatedly, $Y_2 :=\{\, ( \bar{x}_2, \dotsc, \bar{x}_m, a_1, \dotsc,a_{n}) \mid ( a_1, \dotsc,a_{n}, \bar{x}_2, \dotsc, \bar{x}_m )\in Y_1\,\} \overset{.}{\in}  \mathfrak{A}$. Since $h_1$ is an $\mathfrak{A}$-operation, it follows that $Y_3 :=\{\, ( \bar{x}_2, \dotsc, \bar{x}_m, a_1, \dotsc,a_{n-1}, \bar{x}_1 ) \mid ( \bar{x}_2, \dotsc, \bar{x}_m, a_1,\dotsc,a_{n-1},h_1(\bar{x}_1) )\in Y_2\,\} \overset{.}{\in}  \mathfrak{A}$. Finally, $\{\, ( a_1, \dotsc,a_{n-1},\bar{x}_1 , \bar{x}_2, \dotsc, \bar{x}_m) \mid (\bar{x}_2, \dotsc, \bar{x}_m, a_1, \dotsc,a_{n-1},\bar{x}_1 )\in Y_3\,\} \overset{.}{\in}  \mathfrak{A}$.
It is easy to verify that this last set is equal to $W$. 
\end{proof}

\begin{theorem}\label{1207d}
Suppose $\mathfrak{A}=\{ \mathcal{A}_n  \}_{n\geq 1}$ is admissible.
If $\mathcal{F}$ is a collection of $\mathfrak{A}$-operations, then every orderly term over $\mathcal{F}$ is an $\mathfrak{A}$-operation.
\end{theorem}

\begin{proof}
The proof goes by induction on the generation of the orderly terms over $\mathcal{F}$. The identity function on $\omega$ trivially is an $\mathfrak{A}$-operation while every operation in $\mathcal{F}$ is an $\mathfrak{A}$-operation by the assumption. Suppose that $f( \bar{x}_1, \dotsc, \bar{x}_m)=g(h_1(\bar{x}_1), \dotsc, h_m(\bar{x}_m))$ for some orderly terms  $g, h_1, \dotsc, h_m$ over $\mathcal{F}$. By the induction hypothesis, $g, h_1, \dotsc, h_m$ are $\mathfrak{A}$-operations. 
Suppose $n \geq 1$ and $X\in \mathcal{A}_n$. We want
$W:=\{\,( a_1, \dotsc,a_{n-1}, \bar{x}_1, \dotsc, \bar{x}_m) \mid ( a_1, \dotsc,a_{n-1}, f( \bar{x}_1, \dotsc, \bar{x}_m) ) \in X\,\}\overset{.}{\in}  \mathfrak{A}$.
Since $g$ is an $\mathfrak{A}$-operation, it follows that
$Y:= \{\, (a_1, \dotsc,a_{n-1}, b_1, \dotsc, b_m ) \mid ( a_1, \dotsc,a_{n-1}, g(b_1, \dotsc, b_m) ) \in X\,\}\in \mathcal{A}_{n+m-1}$. Therefore, Lemma~\ref{1207c} implies that $\{\,( a_1, \dotsc,a_{n-1}, \bar{x}_1, \dotsc, \bar{x}_m) \mid (a_1,\dotsc,a_{n-1}, h_1(\bar{x}_1), \dotsc, h_m(\bar{x}_m)) \in Y\,\}\overset{.}{\in}  \mathfrak{A}$. This last set is equal to $W$. 
\end{proof}

We introduce some additional notations for the sake of brevity in the next lemma. Suppose $\mathfrak{A}$ is admissible and  $\bar{U}_i$ is a list of 
$\mathfrak{A}$-ultrafilters $U_{1, i}, \dotsc, U_{p_i,i}$ for each $1\leq i\leq n$.
Let $\otimes(\bar{U}_1, \dotsc, \bar{U}_n)$ denote $U_{1,1} \otimes \dotsb \otimes U_{p_1,1} \otimes \dotsb \otimes U_{1,n} \otimes \dotsb \otimes U_{p_n,n}$. 



\begin{lemma}\label{2606c}
Suppose $\mathfrak{A}=\{ \mathcal{A}_n  \}_{n\geq 1}$ is admissible.
Suppose $m \geq 1 $ and  $h_1, \dotsc, h_m$ are $\mathfrak{A}$-operations. 
Then $X \in {h_1}^{\mathfrak{A}}_*(\bar{U}_1)\otimes \dotsc\otimes {h_m}^{\mathfrak{A}}_*(\bar{U}_m)$
if and only if $\{\, (\bar{x}_1, \dotsc, \bar{x}_m) \mid (h_1(\bar{x}_1), \dotsc, h_m(\bar{x}_m))\in X \,\} \in \otimes(\bar{U}_1, \dotsc, \bar{U}_m)$ for each $X \in \mathcal{A}_m$
and lists $\bar{U}_1, \dotsc, \bar{U}_m$ of $\mathfrak{A}$-ultrafilters of the correct length.
\end{lemma}

\begin{proof}
The proof goes by induction on $m$. The base case $m=1$ holds  by the definition of ${h_1}^{\mathfrak{A}}_*$. Assume $m >1$. 
Suppose $X \in \mathcal{A}_m$ and $\bar{U}_1, \dotsc, \bar{U}_m$ are lists of $\mathfrak{A}$-ultrafilters of the correct length.
Let $Z$ be $\{\, (\bar{x}_1, \dotsc, \bar{x}_m ) \mid(h_1(\bar{x}_1), \dotsc, h_m(\bar{x}_m) ) \in X\,\}$.
By Lemma~\ref{1207c}, $Z\overset{.}{\in} \mathfrak{A}$. 
By definition, $X \in {h_1}^{\mathfrak{A}}_*(\bar{U}_1)\otimes \dotsc\otimes {h_m}^{\mathfrak{A}}_*(\bar{U}_m)$
if and only if $\{\, (b_1, \dotsc, b_{m-1}) \mid \{\, b_m \mid (b_1, \dots, b_m) \in X \,\} \in   {h_m}^{\mathfrak{A}}_*(\bar{U}_m)\,\} \in {h_1}^{\mathfrak{A}}_*(\bar{U}_1)\otimes \dotsc\otimes {h_{m-1}}^{\mathfrak{A}}_*(\bar{U}_{m-1})$.
By the induction hypothesis, this holds if and only if $\{\, (\bar{x}_1, \dotsc, \bar{x}_{m-1}) \mid \{\, b_m \mid (h_1(\bar{x}_1), \dotsc, h_{m-1}(\bar{x}_{m-1}), b_m)\in X \,\} \in {h_m}^{\mathfrak{A}}_*(\bar{U}_m)\,\}  \in$ \linebreak
$\otimes(\bar{U}_1, \dotsc, \bar{U}_{m-1})$.
In turn, this holds if and only if
$\{\, (\bar{x}_1, \dotsc, \bar{x}_{m-1}) \mid \{\, (\bar{x}_m) \mid (\bar{x}_1, \dotsc, \bar{x}_m) \in Z\,\} \in \bar{U}_m\,\}
\in \otimes(\bar{U}_1, \dotsc, \bar{U}_{m-1})$.
By Lemma~\ref{0913a}, this holds if and only if $Z \in \otimes(\bar{U}_1, \dotsc, \bar{U}_m)$.
\end{proof}


\begin{theorem}\label{1208h}
Suppose $\mathfrak{A}=\{ \mathcal{A}_n  \}_{n\geq 1}$ is admissible and  $\mathcal{F}$ is a collection of $\mathfrak{A}$-operations.
If $U$ is an $\mathfrak{A}$-ultrafilter idempotent for $\mathcal{F}$ with respect to $\mathfrak{A}$, then $U$ is idempotent for the collection of orderly terms over $\mathcal{F}$ with respect to $\mathfrak{A}$.
\end{theorem}

\begin{proof}
By Lemma \ref{1207d}, every orderly term $f$ over $\mathcal{F}$ is indeed an $\mathfrak{A}$-operation 
The proof goes by induction on the generation of the orderly terms over $\mathcal{F}$. If $I$ is the identity function on $\omega$, then $I^{\mathfrak{A}}_*$ is the identity function on $\beta\mathfrak{A}$ and hence $I^{\mathfrak{A}}_*(U)=U$. If $f \in \mathcal{F}$, then $U$ is idempotent for $f$ with respect to $\mathfrak{A}$ by the assumption. Suppose that $f( \bar{x}_1, \dotsc, \bar{x}_m)=g(h_1(\bar{x}_1), \dotsc, h_m(\bar{x}_m))$ for some orderly terms $g, h_1, \dotsc, h_m$ over $\mathcal{F}$. By the induction hypothesis,  $g^{\mathfrak{A}}_*(U,\dotsc,U)=U$ and ${h_i}^{\mathfrak{A}}_*(U,\dotsc,U)=U$ for each $1\leq i\leq m$. 
It suffices to show that $f^{\mathfrak{A}}_*( U,\dotsc, U)=g^{\mathfrak{A}}_*({h_1}^{\mathfrak{A}}_*(U,\dotsc, U  ), \dotsc, {h_m}^{\mathfrak{A}}_*(U,\dotsc, U))$. 
Suppose $X \in \mathcal{A}_1$. 
 By Lemma~\ref{2606c}, $g^{-1}[X] \in {h_1}^{\mathfrak{A}}_*(U,\dotsc, U)\otimes \dotsc\otimes {h_m}^{\mathfrak{A}}_*(U,\dotsc, U)$
 if and only if 
$\{\, (\bar{x}_1, \dotsc, \bar{x}_m) \mid (h_1(\bar{x}_1), \dotsc, h_m(\bar{x}_n))\in g^{-1}[X] \,\} \in U\otimes \dotsb\otimes U$
and so if and only if $f^{-1}[X] \in U\otimes \dotsb\otimes U$.
\end{proof}


\section{Open problems}

Our introduced framework opens up a wide range of possible research problems.
Here, only some problems that interest us and are most relevant to this paper are mentioned.


Suppose $S_n$ is a countable collection of subsets of $\omega^n$ for every $n \geq 1$.
The construction in the proof of Theorem~\ref{0905b} can be modified to give us a nontrivial admissible $\mathfrak{A}=\{ \mathcal{A}_n  \}_{n\geq 1}$ such that $\mathcal{A}_n$ is a countable superset of $S_n$ for every $n\geq 1$ and there exists a nonprincipal $\mathfrak{A}$-ultrafilter.
In view of Example~\ref{1912a}, it is natural to ask whether 
this construction can be refined further to ensure that every ultrafilter on $\mathcal{A}_1$ is an $\mathfrak{A}$-ultrafilter. 
On the other hand, whether there exists an admissible $\mathfrak{A}$ such that only the principal ultrafilters are $\mathfrak{A}$-ultrafilters seems to be interesting and not obvious.

Suppose $\mathfrak{A}$ is admissible.  If $f$ is a binary associative $\mathfrak{A}$-operation, an ultrafilter idempotent for $f$ with respect to $\mathfrak{A}$ need not exist. On the other hand, Example~\ref{2706e} shows that such ultrafilters may exist even if the algebra is not Ramsey. These unexpected phenomena may be explained by the Ramseyness of the algebra from the point of view of $\mathfrak{A}$. Hence, one can try to formulate the notion of $\mathfrak{A}$-Ramsey algebra. With the right formulation, one might prove that whenever
$(\omega,f)$ is an $\mathfrak{A}$-Ramsey algebra, there exists an $\mathfrak{A}$-ultrafilter idempotent for $f$ with respect to $\mathfrak{A}$, at least in some countable setting.

Finally, we emphasize that Carlson's original problem of the existence of idempotent ultrafilters for every Ramsey algebra is the ultimate goal. The work initiated here is the starting point for continuation work towards that direction as well as motivating others to do so.

\section*{Acknowledgements}

This paper grows partially out of the author's thesis submitted as a partial fulfilment for the award of PhD to the Ohio State University under the supervision of Timothy Carlson.  The author would like to thank Carlson for introducing the concept of Ramsey algebras and for his great mentorship. The open problem regarding the existence of idempotent ultrafilters for every Ramsey algebra was conveyed by him during the supervision.

\section*{Conflict of Interests}

The author declares that there is no conflict of interests
regarding the publication of this article.


\end{document}